\documentclass[11pt]{article}

\usepackage{amssymb}
\usepackage{amsmath}
\usepackage{amsfonts}
\usepackage{eucal}
\usepackage{graphicx}

\newtheorem{theorem}{Theorem}[section]

\newtheorem{lemma}[theorem]{Lemma}
\newtheorem{rem}[theorem]{Remark}

\newenvironment{definition}[1][Definition]{\begin{trivlist}
\item[\hskip \labelsep {\bfseries #1}]}{\end{trivlist}}
\newenvironment{proof}[1][Proof]{\begin{trivlist}
\item[\hskip \labelsep {\bfseries #1}]}{\end{trivlist}}

\begin{document}

\title{Stochastic Stefan Problems Driven By Standard Brownian Sheets}
\author{Zhi Zheng, Richard B. Sowers}
\date{July 1, 2012}

\maketitle

\begin{abstract}
In this paper we study the effect of stochastic perturbations on a common type of moving boundary value PDE's which endorse Stefan boundary conditions, or Stefan problems, and show the existence and uniqueness of the solutions to a number of stochastic equations of this kind. Moreover we also derive the space and time regularities of the solutions and the associated boundaries via Kolmogorov's Continuity Theorem in an appropriately defined normed space.

The paper first conveys our previous results where randomness is smoothly correlated in space and Brownian in time, then introduces a new methodology that enables us to prove the existence and uniqueness of the solution to the standard heat equation driven by a space-time Brownian sheet as well as its boundary regularity, and finally extends it to the stochastic moving boundary partial PDE's driven by the same type of randomness.

\end{abstract}

\section{Introduction}

An important type of problems in the theory of partial differential equations (PDE's) is the moving boundary value problems. In this paper we study the effect of stochastic perturbations on such type of problems with Stefan boundary conditions, or Stefan problems, which have various applications in physics, engineering, and finance, and show the existence and uniqueness of the solutions to a number of stochastic equations of such kind. We also obtain the space and time regularities of the solutions and their associated boundaries via Kolmogorov's Continuity Theorem in a defined normed space.

The paper first conveys our previous results where randomness is smoothly correlated in space and Brownian in time, which has the convenience that the boundary regularity is naturally satisfied, then introduces a new methodology that enables us to prove the existence and uniqueness of the solution to the standard heat equation driven by a space-time Brownian sheet as well as its boundary regularity by defining an appropriate normed space whose norm correctly curvatures both the decay of the solution and its regularity at the boundary, and finally extends it to the stochastic moving boundary partial PDE's driven by the same type of randomness, where the new normed space is defined essentially in the same manner as enlightened by the previous result.

\subsection{Background}

\begin{figure}
\centering
\includegraphics[width=100mm]{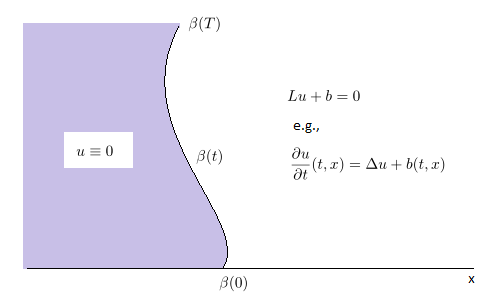}
\caption{An Illustration of Moving Boundary PDE's}
\label{Fig:mbp1}
\end{figure}

A moving boundary PDE of $u(t,x)$ describes the behavior of a system that consists of two phases, as illustrated in Figure~\ref{Fig:mbp1}, where $\beta(t)$ is a moving boundary which is part of the solution and must be solved simultaneously with $u(t,x)$. As can be seen in Figure~\ref{Fig:mbp1}, in the region to the left of the moving boundary (namely, the set $\{(t,x)\in[0,\infty)\times\mathbb R:x\le\beta(t)\}$) $u$ is constantly set to $0$; on the right side of the boundary (the set $\{(t,x)\in[0,\infty)\times\mathbb R:x>\beta(t)\}$) $u$ is described by a PDE of the general form $Lu+b=0$ where $L$ is a predefined second-order differential operator.

 For a moving boundary PDE problem, in addition to the regular boundary condition such as the Dirichlet condition  $u(t,\beta(t))=0$, there is always an extra boundary condition that describes the dynamics at the moving boundary, for instance, the \emph{Stefan boundary condition}
 \begin{equation}\label{Eq:SCondition}
  \frac{\partial u}{\partial x}(t,\beta(t)+)=\rho\dot\beta(t).
  \end{equation}

The type of moving boundary PDE's we consider throughout the thesis is the \emph{Stefan problems}, where $L$ is a heat or parabolic operator (for instance, $L:=-\partial/\partial t+\partial^2/\partial x^2$) with the Stefan boundary condition. Such type of problems has a variety of applications. For instance, in physics, they model the phenomena such as ice melting with the Stefan condition describing the heat balance at the interface (the moving boundary, see~\cite{fri64}); in finance they model the valuation of American options with the PDE derived from the Black-Scholes formula and the moving boundary describing the early exercise price boundary (see Lemma 7.8, Chapter 2,~\cite{mmf}).

\subsection{Motivation and Results}

In the mathematics of this thesis we are interested in the stochastic versions of the Stefan problems, namely, $b(t,x)$ is a formal notation about the stochastic addition (noise). In general $b(t,x)$ is a 2-dimensional distribution and therefore we work within the framework of the stochastic PDE theory by Walsh in~\cite{walsh86}, and based on the weak formulations of the equations and their equivalent evolution equations.

When $b(t,x)$ is multiplicative, namely, $b(t,x)=u(t,x)\dot W(t,x)$ where $\dot W$ is the noise (formal), ~\cite{kms} proves the existence and uniqueness of the solution when $\dot W=\dot W(t)$ is a distribution (Brownian) only in time and is constant in space. Also, in~\cite{kzb} we proved the existence and uniqueness of the solution when $\dot W(t,x)$ is a distribution (Brownian) only in time and is smoothly correlated (``colored'') in space, which is quickly reviewed in Section~\ref{Ch:sspcolored}.

However, when $\dot W(t,x)$ is a distribution (Brownian) in both space and time, we face a number of novel challenges that are beyond the scope of current literature:
\begin{enumerate}
\item[(1)] The spatial derivatives of the solution may not exist (see~\cite{walsh86}), which means the techniques we used in~\cite{kzb} based on $H$-norms may not be available; more importantly, the Stefan condition~\eqref{Eq:SCondition} involves the spatial derivative of the solution at the boundary, and therefore we shall show the existence of such a derivative simultaneously with the existence and uniqueness of the solution.
\item[(2)] The stochastic perturbation given by $b(t,x)$ is no longer spatially Lipschitz as in~\cite{kms} and~\cite{kzb}, which means in order to control the boundary shift effect in the It\^{o} integrals and the nonlinear drift term we may need additional spatial regularities on $b(t,x)$ in addition to just being multiplicative as in~\cite{kms} and~\cite{kzb}; in other words, although the perturbation $b$ vanishes when $u=0$, $u$ itself may not provide sufficient spatial smoothness to control the It\^{o} integrals when the boundary shifts.
\end{enumerate}

To tackle (1) alone, in Section~\ref{Ch:sheat} we first study the stochastic heat equation driven by a multiplicative space-time Brownian noise, namely, with $W$ a standard 2-dimensional Brownian sheet defined in~\cite{walsh86},
\[\begin{split}
\frac{\partial u}{\partial t}(t,x)&=\frac{\partial^2 u}{\partial x^2}(t,x)+u(t,x)\frac{\partial^2 W}{\partial t\partial x},\forall x>0,t\in[0,T],\\
u(0,x)&=u_0(x),\\
u(t,0)&=0,\forall x\le 0,t\in[0,T].
\end{split}\]
From such study we obtain the conditions and form of norms under which a boundary derivative exists, and more importantly, develop the essential techniques to calculate the sample-wise space and time modulus of continuity of the solution by means of Kolmogorov's Continuity Theorem (see~\cite{bm}), which is critical to evaluate the regularity needed to control the boundary shift effect.

Next, in Section~\ref{Ch:sspw}, first with the same multiplicative space-time Brownian noise (that is, $b(t,x)=uW_{tx}$,) we preliminarily calculate the effect of a boundary shift on the It\^{o} integral using the techniques and results obtained in the previous section, and find that it may be difficult to obtain the wanted control on iteration. Therefore we make the change in the stochastic perturbation $b(t,x)$ such that it is even smoother in space and in turn study the following stochastic Stefan problem driven by a scaled space-time Brownian noise:
\[\begin{split}
 \frac{\partial u}{\partial t}&=\frac{\partial^2 u}{\partial t^2}+ \sigma(x-\beta(t))\frac{\partial^2 W}{\partial t\partial x},\forall x>\beta(t),\\
u(t,x)&=0,\forall x\le \beta(t),\\
u(0,x)&=u_0(x), \\
\rho \dot \beta(t) &=\lim_{x\searrow \beta(t)} \frac{u(t,x)}x,\mathbb P\textrm{-a.s.}\\
 \end{split}\]
 where $b(t,x)=\sigma(x-\beta(t))W_{tx}$ and $\sigma(\cdot)$ is a function that satisfies certain regularity conditions which are sufficient to tackle (2). We show the existence and uniqueness of the solution to the above equation and additionally the regularity of the boundary using results from this and previous sections. This result also serves as the mathematical foundation for the modeling of the dynamics of limit orders.

\section{A Stochastic Stefan Problem with Spatially Colored Noise}\label{Ch:sspcolored}

In this section we quickly review our work in~\cite{kzb} by giving the main theorems and lemmas without proofs. We studied a stochastic Stefan problem of $u(t,x)$ driven by a multiplicative noise $u(t,x)d\xi_t(x)$, where $\xi_t(x)$ is a noise that is Brownian in time and smoothly correlated (or ``colored'') in space. Specifically, fix a probability space $(\Omega,\mathcal F, \mathbb P)$, and suppose $W:\Omega\times\mathbb R\times\mathbb R\to \mathbb R$ is the standard 2-dimensional Brownian sheet. Suppose also that $\eta:\mathbb R\to\mathbb R$ is $C^\infty$ and $\|\eta\|_{L^2(\mathbb R)}=1$. Then for $t\ge 0, x\in\mathbb R$, define
\[ \xi_t(x):=\int_0^t\int_0^\infty \eta(x-y)W(dyds).\]
Then we have the following problem and results.

\subsection{Problem and Main Theorem}

We showed the existence and uniqueness of the solution $u(t,x)$ to the following formal equation:
\begin{equation}\label{Eq:SSPC}
\begin{split}
 \frac{\partial u}{\partial t}&=\frac{\partial^2 u}{\partial t^2}+ u(t,x)d\xi_t(x),\forall x>\beta(t),\\
u(t,x)&=0,\forall x\le \beta(t),\\
u(0,x)&=u_0(x), \\
\rho \dot \beta(t) &=\lim_{x\searrow \beta(t)} \frac{u(t,x)}x, \mathbb P\textrm{-a.s.}\\
\end{split}
\end{equation}
  for $0\le t<\tau$ where $\tau$ is some well defined stopping time.

Since $W$ is a distribution, equation~\eqref{Eq:SSPC} is in fact formal, and we need to work on the weak definition (Definition 3.2 in~\cite{kzb}) and its equivalent evolution equation (Equation (15) or Lemma 3.4 in~\cite{kzb}). The main theorem is
\begin{theorem}\label{Thm:SSPC} The solution $u(t,x),\beta(t)$ to~\eqref{Eq:SSPC} exists and is unique for for $0\le t <\tau:=\lim_{L\to\infty}\tau^L$ where $\tau^L:=\inf\{t\in\mathbb R_+\cup\{0\}: |\dot \beta(t)|\ge L\}$, and $\tilde u(t,x):=u(t,x+\beta(t))$ satisfies
\begin{equation}\label{Eq:SSPCU}
\begin{split}
\tilde u(t,x)&=\int_0^\infty p(t,x,y)u_0(y)dy\\
&\quad\quad+\int_0^t\int_0^\infty \dot\beta(s)q(t-s,x,y)\tilde u(s,y)dyds\\
&\quad\quad+\int_0^t\int_0^\infty p(t-s,x,y)u(s,y) d\xi_t(x)ds,
\end{split}\end{equation}
where $p,q$ are standard kernels defined in~\cite{kzb}.
\end{theorem}

\subsection{A Transformation}

We make the natural transformation $\tilde u(t,x):=u(t,x+\beta(t))$, which transforms the original weak definition and its equivalent evolution equation to a nonlinear PDE in the fixed domain $[0,T]\times[0,\infty)$. Then we have
\begin{lemma} The weak solution $u(t,x)$ is obtained by getting it from $\tilde u(t,x)$ of~\eqref{Eq:SSPCU} by the transformation $u(t,x):=\tilde u(t,x-\beta(t))$, where
\[ \dot\beta(t)=\frac1\rho \frac{\partial \tilde u}{\partial x}(t,0+).\]
\end{lemma}

This technique is also used in Section~\ref{Ch:sspw}.

\subsection{Existence and Uniqueness of the Truncated Solution}

First, in order to control the nonlinear drift term we shall truncate the $H^2$-norm of the solution (which is shown as equivalent to truncating the nonlinear term, or $\dot\beta(t)$), and work with the truncated solution $\tilde u^L(t,x)$, namely, the solution to
 \begin{equation}\label{Eq:SSPCUL}
\begin{split}
\tilde u^L(t,x)&=\int_0^\infty p(t,x,y)u_0(y)dy\\
&\quad\quad+\int_0^t\int_0^\infty \dot\beta(s)q(t-s,x,y)\Psi_L\left(\|\tilde u^L(s,\cdot)\|_H\right)\tilde u^L(s,y)dyds\\
&\quad\quad+\int_0^t\int_0^\infty p(t-s,x,y)u^L(s,y) d\xi_t(x)ds,
\end{split}\end{equation}
where $\Psi_L:[0,\infty)\to [0,1]$ is as a smooth monotone decreasing function that satisfies $\chi_{[0,L]}\le\Psi_L\le\chi_{[0,L+1]}$.

 The existence and uniqueness of the truncated solution is proved by a Picard-type iteration on $H^2$-spaced combined with similar calculations in Lemma 3.3 of~\cite{walsh86}. Unlike in Section~\ref{Ch:sheat} and Section~\ref{Ch:sspw}, in this problem we need not worry about the existence of the spatial derivatives, or in particular,
\[  \frac{\partial \tilde u}{\partial x}(t,0+),\]
which is the right hand side in the Stefan boundary condition, because the stochastic perturbation is smooth in space. By calculations of such an iteration and the structural results about $H^2$-space (see Section 5.1 of~\cite{kzb}), combined with Lemma 3.3 of ~\cite{walsh86}, we have
\begin{lemma} Fix $L>0$. Then the solution $\tilde u^L(t,x)$ to~\eqref{Eq:SSPCUL} exists and is unique.
\end{lemma}

\subsection{Relaxation of the Truncation}

Define the stopping time
\[ \tau^L:=\inf\{t\ge0:\|\tilde u^L(t,\cdot)\|_H\ge L\}.\]
Also, define
\[ \tau:=\lim_{L\to\infty} \tau^L,\]
and
\[ \tilde u(t,x):=\lim_{L\to\infty} \tilde u^L(t,x).\]

We then have
\begin{lemma}
\[ \overline{\lim_{t\nearrow \tau}}\|\tilde u(t,\cdot)\|_H=\infty\]
and
\[ \overline{\lim_{t\nearrow \tau}}\left|\frac{\partial \tilde u}{\partial x}(t,0+)\right|=\infty.\]
\end{lemma}

Finally, we have
\begin{lemma} The solution $\tilde u(t,x)$ to~\eqref{Eq:SSPCU} exists and is unique.
\end{lemma}

Combining all the lemmas, we showed the main theorem. Note that the idea of first stopping $|\dot\beta(t)|$ from growing too large (that is, exceeding a fixed $L$), then using this to control the nonlinear drift term, and finally relaxing this truncation and obtaining a global stopping time $\tau$ is important, and is also used in Section~\ref{Ch:sspw} when we study a stochastic Stefan problem driven by a scaled space-time Brownian noise.

\section{Boundary Regularity of the Stochastic Heat Equation}\label{Ch:sheat}

In the previous section we have shown the existence and uniqueness of a stochastic Stefan problem with a spatially-colored and Brownian-in-time noise. Since we would further study a stochastic Stefan problem with space-time Brownian noise under certain regularity conditions, it is necessary that we first understand the effect of such noise on the regularity of the boundary, namely, the differentiability of the solution at the moving boundary, in addition to the existence and uniqueness of the solution itself. This task is critical because the Stefan boundary condition of such a problem involves the spatial derivative of the solution at the boundary, and since the noise is Brownian in space (as well as in time), the solution in general may not have a spatial derivative everywhere except at the boundary.

Therefore, to simplify the problem, we first in this section consider a stochastic heat equation of $u$ driven by a multiplicative space-time Brownian noise, so that the noise vanishes at the boundary (where $u=0$), and we would expect that $u$ is differentiable just at the boundary. Specifically, by removing the shift effect of the moving boundary and studying a stochastic heat equation of this kind, we look to understand
\begin{enumerate}
\item[(1)] under what sense (or, in what normed space) the solution exists, and the connection between such a norm or space and the differentiability of the solution at the moving boundary;
\item[(2)] in what sense ($\mathbb P$-a.s.? in $L^p$? etc.)  the Stefan boundary condition holds;
\item[(3)] the spatial regularity of the solution (H\"{o}lder continuity? with what parameters?) which may guide us on the study of a Stefan problem with a space-time Brownian noise in the next section, in particular, the effect of a boundary shift on the iteration of the It\^{o} integral.
\end{enumerate}

\subsection{Problem and Main Theorem}
Fix a probability space $(\Omega,\mathcal{F},\mathbb{P})$, and suppose $W:\Omega\times\mathbb R\times\mathbb R\to\mathbb R$ is a standard 2-dimensional Brownian sheet. Consider the (formal) stochastic heat equation of $u(t,x)$ with a multiplicative space-time Brownian noise on $[0,T]\times[0,\infty)$, under a Dirichlet boundary condition:
\begin{equation}\label{Eq:sheatDiff}
\begin{split}
\frac{\partial u}{\partial t}(t,x)&=\frac{\partial^2 u}{\partial x^2}(t,x)+u(t,x)\frac{\partial^2 W}{\partial t\partial x},\forall x>0,t\in[0,T],\\
u(0,x)&=u_0(x),\\
u(t,0)&=0,\forall x\le 0,t\in[0,T].
\end{split}\end{equation}

From the classic work of~\cite{walsh86} by Walsh, the weak solution to the formal equation~\eqref{Eq:sheatDiff} is equivalent to the evolution equation
\[ u(t,x)=\int_0^\infty p(t,x,y)u_0(y)dy + \int_0^t\int_0^\infty p(t-s,x,y)u(s,y)W(dyds)\]
where $t\in [0,T],x\ge 0$ and $p(t,x,y)$ is the corresponding kernel as defined in the previous section. Then we have the following theorem as the main conclusion of this section:
\begin{theorem}\label{Thm:sheatMain}
\begin{enumerate}
\item[(1)] The solution $u(t,x)$ to~\eqref{Eq:sheatDiff} exists and is unique with respect to a normed space;
\item[(2)] $\mathbb P$-a.s., for all $t\in[0,T]$, $u(t,x)$ is differentiable at $x=0$, and
\[\frac{\partial u}{\partial x}(t,0)= \lim_{x\searrow 0} \frac{u(t,x)}x = \int_0^\infty \frac{\partial p}{\partial x}(t,0,y)u_0(y)dy + \int_0^t\int_0^\infty \frac{\partial p}{\partial x}(t,0,y) u(s,y)W(dyds);\]
\item[(3)] For $t\in[0,T]$, define $v(t,x):=u(t,x)/x$ for $x>0$ and $v(t,0):=\lim_{x\searrow 0} v(t,x)$, then $\mathbb P$-a.s.,
$v(t,x)$ is $\left(\frac14-\epsilon,\frac16-\epsilon\right)$-H\"{o}lder continuous on $[0,T]\times[0,1]$ for $\epsilon>0$.
\end{enumerate}
\end{theorem}

Theorem~\ref{Thm:sheatMain} is proved in Section~\ref{Sec:ProofsheatMain} using the integral regularities of a newly defined kernel $\tilde p(t,x,y)$ in Section~\ref{Sec:ProofsheatCalc} combined with a newly defined norm and an argument based on Kolmogorov's Continuity Theorem (see~\cite{bm}) and Lemma 3.3 of~\cite{walsh86} in Section~\ref{Sec:ProofsheatMain}.

\subsection{Integral Regularities of Kernel $\tilde p(t,x,y)$}\label{Sec:ProofsheatCalc}

The existence, uniqueness, and regularity of the solution described in Theorem~\ref{Thm:sheatMain} are based on a number of integral regularities of a newly defined kernel $\tilde p(t,x,y)$. We present and prove them in this section.

Define a new kernel $\tilde p(t,x,y)$ as
\[ \tilde p(t,x,y):=\frac yx p(t,x,y),\forall x>0 \]
and
\[ \tilde p(t,0,y):=\lim_{x\searrow 0}\tilde p(t,x,y)=y\frac{\partial p}{\partial x}(t,0,y), \]
then we have

\begin{lemma}\label{Thm:sheatCalc}
\begin{enumerate}
\item[(1)] $\forall x\ge0$,
\[ \int_0^\infty \tilde p^2(s,x,y)dy \le \frac{C}{\sqrt s};\]
\item[(2)] $\forall x,y\in[0,1],t\in[0,T]$,
\[ \int_0^t \int_0^\infty \left[\tilde p(s,x,z)-\tilde p(s,y,z)\right]^2dzds \le C_T(x-y)^{\frac 13};\]
\item[(3)] $\forall x\ge0,s,t\in[0,T]$,
\[ \int_s^t \int_0^\infty \tilde p^2(r,x,y)dydr+ \int_0^s \int_0^\infty \left[\tilde p(r+(t-s),x,y)-\tilde p(r,x,y)\right]^2dydr \le D_T|t-s|^{\frac 12}.\]
\end{enumerate}
\end{lemma}
\begin{proof} We only need to prove the above facts for $[0,T]\times(0,1]$, since by Fatou's Lemma they can be extended to the cases for $x,y=0$.

Throughout the calculations we repeatedly use the following facts:
\begin{enumerate}
\item[(a)] \[ \int_0^\infty y^n \exp\left(-\frac{y^2}s\right)dy=\Gamma\left(\frac{n+1}2\right)s^{\frac{n+1}2};\]
\item[(b)] if $f$ and $g$ are even, then
\[\begin{split}
 &\int_0^\infty g(y)[f(y-x)+f(y+x)]dy=\frac12\int_{-\infty}^\infty g(y)[f(y-x)+f(y+x)]dy\\
 &=\frac12\int_{-\infty}^\infty [g(y-x)+g(y+x)]f(y)dy=\int_0^\infty [g(y-x)+g(y+x)]f(y)dy;
 \end{split}\]
\item[(c)] using the fact that $1-\exp(-x)\le 1\wedge x$, we have for $A>0$,
\[ \int_0^t \frac{1-\exp\left(-\frac{A}{s}\right)}{\sqrt{s}}ds\le\int_0^A \frac{ds}{\sqrt{s}}ds+\int_A^\infty \frac{Ads}{s\sqrt{s}}=4\sqrt{A}.\]
\end{enumerate}
Then
\begin{enumerate}
\item[(1)]
\[\begin{split}
 \int_0^\infty \tilde p^2(s,x,y)dy&=\frac{C_1}{x^2s}\int_0^\infty y^2\left[ e^{-\frac{(x+y)^2}s}+e^{-\frac{(x-y)^2}s}-2e^{-\frac{x^2+y^2}s}\right]dy\\
 &=\frac{C_1}{x^2s}\int_0^\infty \left[2(x^2+y^2)e^{-\frac{y^2}s}-2y^2e^{-\frac{x^2+y^2}s}\right]dy\\
 &=\frac{2C_1}s\int_0^\infty e^{-\frac{y^2}s}dy+\frac{2C_1}{x^2s}\left(1-e^{-\frac{x^2}s}\right)\int_0^\infty y^2e^{-\frac{y^2}s}dy\\
 &\le \frac{2C_1\Gamma\left(\frac12\right)}{\sqrt s}+ \frac{2C_1\Gamma\left(\frac32\right)}{\sqrt s}=\frac{3C_1\Gamma\left(\frac12\right)}{\sqrt s}.
 \end{split}\]
\item[(2)] Suppose $x\le y$, and define $h:=y-x$. Then from above,
\[ \int_0^\infty \left[\tilde p(s,x,z)-\tilde p(s,y,z)\right]^2dz=\frac{C_1\Gamma\left(\frac12\right)}{\sqrt{s}}\left[4+s\left(\frac{1-e^{-\frac{x^2}s}}{x^2}+\frac{1-e^{-\frac{(x+h)^2}s}}{(x+h)^2}\right)\right]-I_X\]
where
\[\begin{split}
 I_X&=\frac{2C_1}{x(x+h)s}\int_0^\infty z^2\left(e^{-\frac{(x-z)^2}{2s}}-e^{-\frac{(x+z)^2}{2s}}\right)\left(e^{-\frac{(x+h-z)^2}{2s}}-e^{-\frac{(x+h+z)^2}{2s}}\right)dz\\
 &=\frac{2C_1}{x(x+h)s}\int_0^\infty z^2\left[e^{-\frac{h^2}{4s}}\left(e^{-\frac{(x+\frac h2-z)^2}{s}}+e^{-\frac{(x+\frac h2+z)^2}{s}}\right)-e^{-\frac{\left(x+\frac h2\right)^2}s}\left(e^{-\frac{\left(z+\frac h2\right)^2}s}+e^{-\frac{\left(z-\frac h2\right)^2}s}\right)\right]dz\\
 &=\frac{4C_1}{x(x+h)s}\left\{e^{-\frac{h^2}{4s}}\int_0^\infty \left[z^2+\left(x+\frac h2\right)^2\right]e^{-\frac{z^2}s}dz-e^{-\frac{\left(x+\frac h2\right)^2}s}\int_0^\infty \left[z^2+\frac {h^2}4\right]e^{-\frac{z^2}s}dz\right\}\\
 &=\frac{2C_1\Gamma\left(\frac12\right)}{x(x+h)\sqrt{s}}\left\{e^{-\frac{h^2}{4s}}\left[s+2\left(x+\frac h2\right)^2\right]-e^{-\frac{\left(x+\frac h2\right)^2}s}\left[s+2\left(\frac {h^2}4\right)\right]\right\}\\
 &=\frac{C_1\Gamma\left(\frac12\right)}{\sqrt{s}}e^{-\frac{h^2}{4s}}\left\{4+\frac{1-e^{-\frac{x(x+h)}s}}{x(x+h)}(2s+h^2)\right\}\ge\frac{C_1\Gamma\left(\frac12\right)}{\sqrt{s}}e^{-\frac{h^2}{4s}}\left\{4+2s\frac{1-e^{-\frac{x(x+h)}s}}{x(x+h)}\right\}.
 \end{split}\]
Note that from (c) above, we get
\[ \int_0^t \frac{1-e^{-\frac{h^2}{4s}}}{\sqrt{s}}ds \le 2h.\]
Then we need to consider
\[ J(t,x,h):=\int_0^t \left\{ \sqrt{s}\left[\frac{1-e^{-\frac{x^2}s}}{x^2}+\frac{1-e^{-\frac{(x+h)^2}s}}{(x+h)^2}-2e^{-\frac{h^2}{4s}}\left(\frac{1-e^{-\frac{x(x+h)}s}}{x(x+h)}\right)\right]\right\}ds.\]
Now we use two methods to bound $J(t,x,h)$.
\begin{enumerate}
\item[(I)]Using the fact that $x-\frac{x^2}2\le 1-\exp(-x)\le x$, we get
\[J(t,x,h)\le\int_0^t \frac{2}{\sqrt{s}}\left(1-e^{-\frac{h^2}{4s}}\right)ds+x(x+h)\int_0^t\frac{e^{-\frac{h^2}{4s}}}{s\sqrt{s}}ds\le4h+\Gamma\left(\frac12\right)\frac{x(x+h)}h.\]
\item[(II)]Define a function
\[ \Phi(t,x):=\int_0^t\frac{1-\exp\left(-\frac{x^2}s\right)}{x^2}ds=\frac{t}{x^2}-\int_{\frac{x^2}t}^\infty u^{-2}e^{-u}du.\]
Then
\[ \left|\frac{\partial \Phi}{\partial x}(t,x)\right|=2t\frac{1-e^{-\frac{x^2}t}}{x^3}\le \frac 2x.\]
Then we have
\[\begin{split}
 J(t,x,h)&\le \left[\Phi(t,x)+\Phi(t,x+h)-2\Phi\left(t,x+\frac h2\right)\right]+\frac{2}{x(x+h)}\int_0^t\sqrt{s}\left(1-e^{-\frac{h^2}{4s}}\right)ds\\
 &\le \frac h2 \left(\frac 2x+\frac 2{x+\frac h2}\right)+ \frac{h^2\sqrt{t}}{x(x+h)}\le\frac{2h}{x}+ \frac{h^2\sqrt{t}}{x(x+h)}.
 \end{split}\]
\end{enumerate}
Combining (I) and (II), we have that for some $C_T>0$, when $x\le h^{2/3}$, using (I) and we get $\int_0^\infty \left[\tilde p(s,x,z)-\tilde p(s,y,z)\right]^2dz\le C_T h^{1/3}$; when $x> h^{3/2}$, using (II) and we get $\int_0^\infty \left[\tilde p(s,x,z)-\tilde p(s,y,z)\right]^2dz\le C_T h^{1/3}$.

\item[(3)] Suppose $s\le t$, and define $k:=t-s$. The first integral on the left hand side is bounded by $C/\sqrt{k}$ using (1). Now, from (1),
\[ \int_0^s \int_0^\infty \left[\tilde p(r+k,x,y)-\tilde p(r,x,y)\right]^2dy=3C_1\Gamma\left(\frac12\right)\left(\sqrt{s+k}-\sqrt{k}+\sqrt{s}\right)-I\]
where (defining $k':=\frac{rk}{2r+k}$)
\[\begin{split}
 I&=\int_0^s \frac{2C_1}{x^2\sqrt{r(r+k)}}\int_0^\infty y^2\left[e^{-\frac{(x-y)^2}{2(r+k)}}-e^{-\frac{(x+y)^2}{2(r+k)}}\right]\left[e^{-\frac{(x-y)^2}{2r}}-e^{-\frac{(x+y)^2}{2r}}\right] dydr \\
 &=\int_0^s \frac{2C_1}{x^2\sqrt{r(r+k)}}\int_0^\infty y^2\left[e^{-\frac{(x-y)^2}{r+k'}}+e^{-\frac{(x+y)^2}{r+k'}}-e^{-\frac{x^2}{r+\frac k2}}\left(e^{-\frac{\left(y-\frac{k}{2r+k}x\right)^2}{r+k'}}+e^{-\frac{\left(y+\frac{k}{2r+k}x\right)^2}{r+k'}}\right)\right]dydr\\
 &=\int_0^s \frac{4C_1}{x^2\sqrt{r(r+k)}}\left[\int_0^\infty (x^2+y^2)e^{-\frac{y^2}{r+k'}}dy-e^{-\frac{x^2}{r+\frac k2}}\int_0^\infty\left(\frac{k^2}{(2r+k)^2}x^2+y^2\right)e^{-\frac{y^2}{r+k'}} dy\right]dr\\
 &=2C_1\Gamma\left(\frac12\right)\int_0^s\frac{dr}{\sqrt{r(r+k)}}\left\{2\left[1-\frac{k^2}{(2r+k)^2}e^{-\frac{x^2}{r+\frac k2}}\right]\sqrt{r+k'}+\frac{1-e^{-\frac{x^2}{r+\frac k2}}}{x^2}(r+k')^{\frac32}\right\}\\
 &\ge3C_1\Gamma\left(\frac12\right)\int_0^s \frac1{r+\frac k2}\frac{r(r+k)}{\left(r+\frac k2\right)^\frac32}dr\ge3C_1\Gamma\left(\frac12\right)\int_0^s \frac{rdr}{(r+k)^\frac32}.
 \end{split}\]
Therefore we have for some $D_T>0$
\[ \int_0^s \int_0^\infty \left[\tilde p(r+k,x,y)-\tilde p(r,x,y)\right]^2dy\le D_T k^\frac12. \Box\]
\end{enumerate}
\end{proof}

\subsection{Proof of the Main Theorem}\label{Sec:ProofsheatMain}

Theorem~\ref{Thm:sheatMain} is proved in two steps. First we show that the solution to~\eqref{Eq:sheatDiff} exists and is unique in a normed space, where the norm is defined so that in the second step the regularity results of the solution can be derived by using the defined norm and the integral regularities of the kernel $\tilde p(s,x,y)$ in Lemma~\ref{Thm:sheatCalc}. In other words, the norm we defined in the next part characterizes the essential component from which we derive the desired H\"{o}lder continuity of the solution and its differentiability at the boundary, which are shown by using Kolmogorov's Continuity Theorem in the second step.

\subsubsection{Existence and Uniqueness of the Solution}
Existence and uniqueness of the solution is shown by a Picard-type iteration. Consider the iteration
\begin{equation}\label{Eq:sheatItrU} u_{n+1}(t,x)=\int_0^\infty p(t,x,y)u_0(y)dy + \int_0^t\int_0^\infty p(t,x,y)u_n(s,y)W(dyds).
\end{equation}
Or, if we define
\[ v_n(t,x):=\frac{u_n(t,x)}x, v_0(x):=\frac{u_0(x)}x,\]
then
\begin{equation}\label{Eq:sheatItrV}
 v_{n+1}(t,x)=\int_0^\infty \tilde p(t,x,y)v_0(y)dy + \int_0^t\int_0^\infty \tilde p(t,x,y)v_n(s,y)W(dyds).
 \end{equation}
Fix $p\ge 1$. Suppose $\{f(x)\}_{x\ge 0}$ is a stochastic process. Define a norm
\[ \|f\|_{2p}:=\sup_{x>0} \mathbb E\left[ f^{2p}(x) \right]. \]

Then we have
\begin{lemma}
For all $p\ge1$, the solution $v(t,x)$ to~\eqref{Eq:sheatItrV} exists and is unique in the space defined by the norm $\|\cdot\|_{2p}$.
\end{lemma}
\begin{proof} We have
\[ \frac{u_{n+1}(t,x)-u_n(t,x)}x= \int_0^t\int_0^\infty \left[ \frac{u_n(s,y)-u_{n-1}(s,y)}y \right]\tilde p(t-s,x,y) W(dyds)\]
or equivalently,
\[ v_{n+1}(t,x)-v_n(t,x)= \int_0^t\int_0^\infty \left[ v_n(s,y)-v_{n-1}(s,y) \right]\tilde p(t-s,x,y) W(dyds).\]

Define $H_n(t):=\|v_n(t,\cdot)-v_{n-1}(t,\cdot)\|_{2p}$.
Then we have from Lemma~\ref{Thm:sheatCalc} (1) that
\[\begin{split}
 H_{n+1}(t) &\le \sup_{x>0} C_p \mathbb E\left[ \int_0^t\int_0^\infty \left( v_n(s,y)-v_{n-1}(s,y) \right)^2p^2(t-s,x,y) dyds\right]^p\\
 &\le \sup_{x>0} C_p \left[ \int_0^t\int_0^\infty p^2(t-s,x,y) dyds\right]^{p-1} \int_0^t\int_0^\infty \mathbb E\left[\left( v_n(s,y)-v_{n-1}(s,y) \right)^{2p}\right]p^2(t-s,x,y) dyds\\
 &\le C_p' t^{\frac{p-1}2} \int_0^t \frac{H_n(s)}{\sqrt{t-s}} ds\le C_{p,T} \int_0^t \frac{H_n(s)}{\sqrt{t-s}} ds.
 \end{split}\]
 where the first inequality comes from Burkholder inequality, the second inequality comes from Jensen's inequality that for $p\ge 1$,
 \[ \left( \frac{\int ab}{\int b}\right)^p \le \frac{\int a^pb}{\int b}\Leftrightarrow \left(\int ab\right)^p \le \left(\int b\right)^{p-1}\int a^pb,\]
 and the third inequality comes from Lemma~\ref{Thm:sheatCalc} (1).
By Lemma 3.3 of~\cite{walsh86},
\[ \sum_n H_n(t) <\infty \] and the convergence is uniform on compacts. That is, if $0\le t\le T$, then $\sum_n H_n(t)<C_{p,T}<\infty$ where $C_{p,T}$ is a constant dependent only on $T$. This gives immediately by Picard-type iteration that the solution $v(t,x)$ exists and is unique with respect to $\|\cdot\|_{2p}$. Moreover,
\[ \|v(t,\cdot)\|_{2p}\le \|v_0\|_{2p}+\sum_n H_n(t)<D_{p,T}<\infty. \Box\]
\end{proof}

Define $u(t,x):=xv(t,x)$, then $u(t,x)$ is the unique solution to~\eqref{Eq:sheatItrU}.

\subsubsection{Regularity of the Solution and Differentiability at the Boundary}
In this part we shall prove the differentiability of the solution at the boundary by giving a $\mathbb P$-a.s. limit or the spatial derivative at the boundary. This is shown by using Kolmogorov's Continuity Theorem combined with the integral regularities shown in the previous parts and the existence of the solution under the $\|\cdot\|_{2p}$ norm. In fact, a stronger result is shown, namely, the solution is $\mathbb P$-a.s. H\"{o}lder continuous with parameters $\left(\frac14-\epsilon,\frac16-\epsilon\right)$ on $[0,T]\times[0,1]$ for $\epsilon>0$, which can also be used to evaluate the impact of its spatial regularity on the boundary shift effect required in the stochastic Stefan problems drive by space-time Brownian noise.

\begin{lemma}
$\mathbb P$-a.s., for all $t\in[0,T]$ the solution $v(t,x)$ to~\eqref{Eq:sheatItrV} is continuous at $x=0$ and
\[\frac{\partial u}{\partial x}(t,0)= \lim_{x\searrow 0} v(t,x) = \int_0^\infty \frac{\partial p}{\partial x}(t,0,y)u_0(y)dy + \int_0^t\int_0^\infty \frac{\partial p}{\partial x}(t,0,y) u(s,y)W(dyds).\]
Moreover, for $t\in[0,T]$ define $v(t,0):=\lim_{x\searrow 0} v(t,x)$, then $\mathbb P$-a.s.,
$v(t,x)$ is $\left(\frac14-\epsilon,\frac16-\epsilon\right)$-H\"{o}lder continuous on $[0,T]\times[0,1]$ for $\epsilon>0$.
\end{lemma}
\begin{proof}
We only need to consider the Brownian term. Fix $p\ge 1$. For $t\in [0,T], x>0$, define
\[ I_1(t,x):=\frac 1x \int_0^t\int_0^\infty u(s,y)p(t-s,x,y)W(dyds)=\int_0^t\int_0^\infty v(s,y)\tilde p(t-s,x,y)W(dyds),\]
and
\[ I_1(t,0):=\int_0^t\int_0^\infty u(s,y)\frac{\partial p}{\partial x}(t-s,0,y)W(dyds)=\int_0^t\int_0^\infty v(s,y)\tilde p(t-s,0,y)W(dyds).\]
Then for $t\in [0,T], 0\le x,y\le 1$, we have
\[\begin{split}
 \mathbb E \left[ (I_1(t,x)-I_1(t,y))^{2p} \right] &\le C_p \mathbb E \left[ \int_0^t \int_0^\infty v^2(s,z) \left[\tilde p(t-s,x,z)-\tilde p(t-s,y,z)\right]^2 dzds \right]^p\\
 &\le C_p\left[ \int_0^t \int_0^\infty \left[\tilde p(t-s,x,z)-\tilde p(t-s,y,z)\right]^2 dzds \right]^{p-1}\cdot \\
 &\quad\quad\quad\quad  \int_0^t \int_0^\infty\mathbb E \left[ v^{2p}(s,z) \right] \left[\tilde p(t-s,x,z)-\tilde p(t-s,y,z)\right]^2 dzds\\
 &\le C_pD_{p,T} \left[ \int_0^t \int_0^\infty \left[\tilde p(t-s,x,z)-\tilde p(t-s,y,z)\right]^2 dzds \right]^p\\
 &\le C_{p,T}' (x-y)^{\frac{p}3}
 \end{split}\]
 where the first inequality comes from Burkholder inequality, the second comes from Jensen's inequality, the third comes from the previous part about uniform boundedness of $\|u(t,\cdot)\|_{2p}$ on $[0,T]$, and the last one comes from Lemma~\ref{Thm:sheatCalc} (2).

Similarly, We also have for $0<s\le t\le T, x\in [0,1]$,
\[ \mathbb E \left[ (I_1(t,x)-I_1(s,x))^{2p} \right] \le 2^{2p-1} [I_{11}(s,t,x)+I_{12}(s,t,x)]\]
where
\[\begin{split}
 I_{11}(s,t,x)&:=\mathbb E\left[ \left(\int_0^s \int_0^\infty v(r,y) \left[\tilde p(t-r,x,y)-q_{s-r}(x,y)\right]W (dydr) \right)^{2p}\right]\\
 I_{12}(s,t,x)&:=\mathbb E\left[\left(\int_s^t \int_0^\infty v(r,y) \tilde p(t-r,x,y)W (dydr) \right)^{2p}\right]
 \end{split}\]
by Jensen's inequality $\left(\frac{a+b}2\right)^{2p}\le \frac{a^{2p}+b^{2p}}2$. Now,
\[\begin{split}
 I_{11}(s,t,x)&\le C_p \mathbb E\left[ \left(\int_0^s \int_0^\infty v^2(r,y) \left[\tilde p(t-r,x,y)-\tilde p(s-r,x,y)\right]^2dydr \right)^p\right]\\
 &\le C_p \left[\int_0^s \int_0^\infty \left[\tilde p(t-r,x,y)-\tilde p(s-r,x,y)\right]^2dydr \right]^{p-1}\cdot\\
 &\quad\quad\quad\quad \int_0^s\int_0^\infty \mathbb E \left[ v^{2p}(r,y) \right] \left[\tilde p(t-r,x,y)-\tilde p(s-r,x,y)\right]^2dydr\\
 &\le C_{1,p,T} (t-s)^{\frac p2}
 \end{split}\]
 using Burkholder's inequality, Jensen's inequality, and Lemma~\ref{Thm:sheatCalc} (3). Similarly, using  Burkholder's inequality, Jensen's inequality, and Lemma~\ref{Thm:sheatCalc} (1), we have
\[ I_{12}(s,t,x)\le C_{2,p,T} (t-s)^{\frac p2}.\]

Summing things up, for all $p\ge 1$, we have for $0\le s,t\le T, 0\le x,y\le 1$,
\[ \mathbb E \left[ (I_1(t,x)-I_1(s,y))^{2p} \right]\le M_{p,T}(x-y)^{\frac {p}3} + N_{p,T} (t-s)^{\frac p2}.\]
By Kolmogorov Continuity Theorem, the above calculations imply that $\mathbb P$ almost surely, $I_1(t,x)$ is $(\gamma,\beta)$-H\"{o}lder continuous on $[0,T]\times[0,1]$ for all $\gamma<1/4, \beta<1/6$. This implies that $I_1(t,0)=\lim_{x\searrow 0} I_1(t,x)$, and that $I_1(t,0)$ is $\gamma$-Holder continuous on $[0,T]$ for all $\gamma<1/4$. $\Box$
\end{proof}

\section{A Stochastic Stefan Problem with Space-Time Brownian Noise}\label{Ch:sspw}

In the previous section we considered the stochastic heat equation of $u$ driven by a multiplicative space-time Brownian noise $u\frac{\partial^2 W}{\partial t \partial x}$ and showed the existence and uniqueness of the solution under normed space defined by $\|\cdot\|_{2p}$ for all $p\ge 1$ and the $\mathbb P$-a.s. differentiability at the boundary and as a by-product, the H\"{o}lder continuity of $u$. The proof provides us with important hints on how to prove the existence and uniqueness of a stochastic Stefan problem driven by space-time Brownian noise, since we shall also prove that the Stefan boundary condition holds simultaneously at the moving boundary.

Fix a probability space $(\Omega,\mathcal F, \mathbb P)$, and suppose $W:\Omega\times\mathbb R\times\mathbb R\to \mathbb R$ is the standard 2-dimensional Brownian sheet. Then the stochastic Stefan problem of $u(t,x)$ has the general (formal) form
\[\begin{split}
 \frac{\partial u}{\partial t}&=\frac{\partial^2 u}{\partial t^2}+ \sigma(t,x,u(t,x))\frac{\partial^2 W}{\partial t\partial x},\forall x>\beta(t),\\
u(t,x)&=0,\forall x\le \beta(t),\\
u(0,x)&=u_0(x), \\
\rho \dot \beta(t) &=\lim_{x\searrow \beta(t)} \frac{u(t,x)}x, \mathbb P\textrm{-a.s.}\\
 \end{split}\]
 for $0\le t<\tau$ where $\tau$ is some well defined stopping time.

The introduction of a moving boundary brings about a number of challenges that do not exist in the stochastic heat equation problem, and one of them is to control the boundary shift effect, namely, if we iterate on the boundary $\beta_n(t)$, then $d_n(t):=\beta_{n+1}(t)-\beta_n(t)$ is the shift of the boundary between iterations. We look to control
\[ \int_0^t \int_0^\infty \left[\tilde p(t-s,x,y+d_n(s))\frac{\sigma(s,y+d_n(s),u(s,y+d_n(s)))}{y+d_n(s)}-\tilde p(t-s,x,y)\frac{\sigma(s,y,u(s,y))}y\right] W(dyds)\]
whose variance is
\[ \int_0^t \int_0^\infty \left[\tilde p(t-s,x,y+d_n(s))\frac{\sigma(s,y+d_n(s),u(s,y+d_n(s)))}{y+d_n(s)}-\tilde p(t-s,x,y)\frac{\sigma(s,y,u(s,y))}y\right]^2 dyds.\]

For multiplicative noise where $\sigma(t,x,u)=u$, the variance above in terms of $d_n$ is completely determined by the spatial regularity of $u$ or $v:=u/x$. From the previous section we know that spatially $v$ is continuous only at the boundary and has $\left(\frac16-\epsilon\right)$-H\"{o}lder continuity for $\epsilon>0$ on $[0,1]$. To evaluate the impact of the spatial regularity on the above variance in terms of $d_n$, we let $\sigma(t,x,u):=x$, and a calculation with two methods similar to Lemma~\ref{Thm:sheatCalc} (2) shows that
\[ \int_0^t \int_0^\infty \left[\tilde p(t-s,x,y+d)-\tilde p(t-s,x,y)\right]^2 dyds\le C_T'd^{\frac23}.\]
Although the tightness of the inequality is not justified, it still gives strong evidence that it would be difficult to have an iteration on $d_n$, even if $\sigma$ is as smooth as $\sigma=x$. Therefore, instead of working with a multiplicative noise, we in this section consider a stochastic Stefan problem drive by scaled space-time Brownian noise $\sigma(x)\frac{\partial^2 W}{\partial t\partial x}$ where $\sigma(x)$ satisfies certain regularity conditions:

\begin{definition} A function $\sigma:[0,\infty)\to\mathbb R$ is a \emph{regular scaling function} if $\sigma(x)$ is Lipschitz in $x$ and $\sigma(x)\sim Ax^\alpha$ at $0$ with $\alpha>\frac32$. Equivalently, $|\sigma(x)|\le \min\{Ax^\alpha,Bx\}$ for some $A,B>0$.
\end{definition}

 Note that $Ax^\alpha$ provides enough smoothness to control the boundary shift effect as iterating on $\dot\beta_n(t)$ below and $Bx$ provides enough control as $x\to\infty$ for other parts so that the noise does not grow too large at infinity.

\subsection{Problem and Main Theorem}
Suppose $\sigma: [0,\infty)\to \mathbb R$ is a regular scaling function. Then we consider the stochastic Stefan problem of $u(t,x)$
\begin{equation}\label{Eq:SSPW}
\begin{split}
 \frac{\partial u}{\partial t}&=\frac{\partial^2 u}{\partial t^2}+ \sigma(x-\beta(t))\frac{\partial^2 W}{\partial t\partial x},\forall x>\beta(t),\\
u(t,x)&=0,\forall x\le \beta(t),\\
u(0,x)&=u_0(x), \\
\rho \dot \beta(t) &=\lim_{x\searrow \beta(t)} \frac{u(t,x)}x,\mathbb P\textrm{-a.s.}\\
 \end{split}\end{equation}
for $0\le t<\tau$ for some well defined stopping time $\tau$.

As before we work on a weak formulation of~\eqref{Eq:SSPW} and its transformed equivalent evolution equation. Unlike in the colored noise case, we do not directly have the differentiability of the solution at the boundary, so in the iteration of the boundary $\dot\beta$ we cannot simply let it be the spatial derivative of the solution at the boundary, but the one similar to what was obtained in the stochastic heat equation, and finally we have an extra step to show that the two coincide $\mathbb P$-a.s., or equivalently, the Stefan boundary condition in~\eqref{Eq:SSPW} holds.

\begin{theorem}\label{Thm:SSPWMain} The solution $u(t,x),\beta(t)$ to~\eqref{Eq:SSPW} exists and is unique for $0\le t <\tau:=\lim_{L\to\infty}\tau^L$ where $\tau^L:=\inf\{t\in\mathbb R_+\cup\{0\}: |\dot \beta(t)|\ge L\}$. Moreover, $\mathbb P$-a.s., $\beta\in C^1([0,\tau))$ and $\dot\beta(t)$ is $\left(\frac14-\epsilon\right)$-H\"{o}lder continuous for $\epsilon>0$.
\end{theorem}

Theorem~\ref{Thm:SSPWMain} is proved in Section~\ref{Sec:ProofSSPWMain} using the integral regularities of the new kernels $\tilde p(t,x,y)$ and $\tilde q(t,x,y)$ proved in Section~\ref{Sec:ProofsheatCalc} and Section~\ref{Sec:ProofSSPWCalc}. The regularity result on the moving boundary is proved via using Kolmorogov's Continuity Theorem similarly to the proof given in Section~\ref{Sec:ProofsheatMain}.

\subsection{Integral Regularities of Kernels $\tilde p(t,x,y),\tilde q(t,x,y)$}\label{Sec:ProofSSPWCalc}
As in the previous section, the proof to the main theorem is based on the integral regularities of newly defined kernels. Define two new kernels (where $\tilde p$ is the same as in the previous section)
\[\begin{split}
\tilde p(t,x,y)&:=\frac yx p(t,x,y), \tilde p(t,0,y):=\lim_{x\searrow 0}\tilde p(t,x,y)=y\frac{\partial p}{\partial x}(t,0,y)=\frac{y^2}{t\sqrt{t}}\exp\left[-\frac{y^2}{2t}\right];\\
\tilde q(t,x,y)&:=\frac yx \frac{\partial p}{\partial x}(t,x,y),\tilde q(t,0,y):=\lim_{x\searrow 0}\tilde q(t,x,y)=y\frac{\partial^2 p}{\partial x^2}(t,0,y)=0.
\end{split}\]

Then we have the following integral regularity results about $\tilde q(t,x,y)$, similar to those for $\tilde p(t,x,y)$ as in Lemma~\ref{Thm:sheatCalc}:
\begin{lemma}\label{Thm:SSPWCalc} There exists $C>0, K>0$ such that
\begin{enumerate}
\item[(1)]
\[ \int_0^\infty |\tilde q(s,x,y)|dy <\frac{C}{\sqrt{s}};\]
\item[(2)] for $x,x+h\in [0,1]$,
\[ \int_0^t\int_0^\infty |\tilde q(t-s,x+h,y)-\tilde q(t-s,x,y)|dyds <Kh^\frac16;\]
\item[(3)] for $k\ge 0$,
\[ \int_0^t\int_0^\infty |\tilde q(t+k-s,x,y)-\tilde q(t-s,x,y)|dyds <Kk^\frac14.\]
\end{enumerate}
\end{lemma}
\begin{proof} First we observe that
\[ \tilde q(t,x,y)=\frac{y^2}{xs}p_+(t,x,y)-\frac{y}{s}p(t,x,y)=\frac{y^2}{xs}p_+(t,x,y)-\frac{x}{s}\tilde p(t,x,y)\]
where
\[ p_+(t,x,y):=\frac{C_1}{\sqrt{s}}\left[e^{-\frac{(x-y)^2}{2s}}+e^{-\frac{(x+y)^2}{2s}}\right].\]
Therefore the calculations to prove the above facts are fairly similar to those in Lemma~\ref{Thm:sheatCalc}, hence we only provide sketch here.
\begin{enumerate}
\item[(1)] this can be shown by using facts (a) and (b) in Lemma~\ref{Thm:sheatCalc}, where
\[\begin{split}
g(y)&:=|y|,\\
f(y)&:=|y|\exp\left(-\frac{y^2}{2s}\right),
\end{split}\]
followed by a direct calculation of the integral after applying (b) with $f,g$.
\item[(2)] Similar to Lemma~\ref{Thm:sheatCalc} (2), we still use a two-method approach to estimate the first part containing $p_+$:
\[ \int_0^t\int_0^\infty \frac{y^4}{s^2}\left[\frac{ p_+(t-s,x+h,y)}{x+h}-\frac{p_+(t-s,x,y)}x\right]^2dyds\le C h^{\frac13}.\]
From Lemma~\ref{Thm:sheatCalc} (1) and (c) we also have
\[ \int_0^t\frac{ds}s\int_0^\infty \left|(x+h)\tilde p(t-s,x+h,y)-x\tilde p(t-s,x,y)\right|dy\le C' h.\]
Combining the above two facts, we have there exists $K>0$ such that
\[ \int_0^t\int_0^\infty |\tilde q(t-s,x+h,y)-\tilde q(t-s,x,y)|dyds <Kh^\frac16.\]
\item[(3)] Similar to Lemma~\ref{Thm:sheatCalc} (3), we first estimate the first part containing $p_+$:
\[ \int_0^t\int_0^\infty \frac{y^4}{x^2s^2}\left[ p_+(t+k-s,x,y)-p_+(t-s,x,y)\right]^2dyds\le C k^{\frac12}.\]
From Lemma~\ref{Thm:sheatCalc} (1) and (c) we also have
\[ \int_0^t\frac{xds}s\int_0^\infty \left|\tilde p(t+k-s,x+h,y)-\tilde p(t-s,x,y)\right|dy\le C' k^\frac12.\]
Combining the above two facts, we have there exists $K>0$ such that
\[ \int_0^t\int_0^\infty |\tilde q(t+k-s,x,y)-\tilde q(t-s,x,y)|dyds <Kk^\frac14. \Box\]
\end{enumerate}
\end{proof}

\subsection{Proof of the Main Theorem}\label{Sec:ProofSSPWMain}
As in the previous sections, we work with the weak formulation of the solution and its equivalent evolution equation. Define $\tilde u(t,x):=u(t,x+\beta(t))$, Then the equivalent evolution equation gives
\begin{equation}\label{Eq:SSPWU}
\begin{split}
\tilde u(t,x)&=\int_0^\infty p(t,x,y)u_0(y)dy\\
&\quad\quad+\int_0^t\int_0^\infty \dot\beta(s)q(t-s,x,y)\tilde u(s,y)dyds\\
&\quad\quad+\int_0^t\int_0^\infty p(t-s,x,y)\sigma(y) W_\beta(dyds),
\end{split}\end{equation}
where $W_\beta$ is the Brownian sheet obtained by shifting $W$ spatially by $x\to x+\beta(t)$.
The main theorem is proved by adopting the following strategy:
\begin{enumerate}
\item[(1)] prove the existence and uniqueness of the truncated solution $\tilde u^L(t,x)$ and its corresponding $\beta(t)$ in $\|\cdot\|_{2p}$ for $0\le t<\tau^L$ for a fixed $L>0$ by an argument based on Picard-type iterations;
\item[(2)] based on (1), show that $\mathbb P$-a.s., the Stefan boundary condition
\[ \lim_{x\searrow 0}\frac{\tilde u^L(t,x)}x = \rho \dot\beta(t)\]
holds by using Kolmogorov's Continuity Theorem on the drift term and the Brownian term;
\item[(3)] let $L\to\infty$ and obtain the solution $\tilde u(t,x)$; transform $\tilde u$ back to $u$ and obtain a weak solution of~\eqref{Eq:SSPW}.
\end{enumerate}

\subsubsection{Existence and Uniqueness of the Truncated Solution}

In this section we first show the existence and uniqueness of $\beta$ that satisfies~\eqref{Eq:SSPWBeta}, and then truncate $\dot\beta(t)$ by a fixed $L>0$ so that the drift term of the solution is controlled from growing too large. Then we show the existence and uniqueness of the solution under the truncation.

Defining $W_\beta$ as before, then we have
\begin{equation}\label{Eq:SSPWBeta}
\begin{split}
\rho\dot\beta(t)&=\int_0^\infty \frac{\partial p}{\partial x}(t,0,y)u_0(y)dy \\
&\quad\quad +\int_0^t\int_0^\infty \frac{\partial p}{\partial x}(t-s,0,y)\sigma(y) W_\beta(dyds),
\end{split}
\end{equation}

\begin{lemma} There exists a unique $\beta(t)$ that satisfies~\eqref{Eq:SSPWBeta}.
\end{lemma}
\begin{proof}
Consider the following iteration on $\beta(t)$:
\[\begin{split}
\rho\dot\beta_n(t)&=\int_0^\infty \frac{\partial p_-}{\partial x}(t,0,y)u_0(y)dy \\
&\quad\quad +\int_0^t\int_0^\infty \tilde p(t-s,0,y)\frac{\sigma(y)}yW_{\beta_{n-1}}(dyds).
\end{split}\]

Fix $p\ge 1$ and define
\[ H_n(t):=\mathbb E\left[(\dot\beta_{n+1}(t)-\dot\beta_n(t))^{2p}\right].\]
Also define
\[ \mathbf g(s,y):=\tilde p(s,0,y)y^{\alpha-1}.\]
Then letting $d_n(t):=\beta_n(t)-\beta_{n-1}(t)$, we have
\[\begin{split}
 H_n(t) &\le K_1\mathbb E\left\{\int_0^t\int_0^\infty \left[\mathbf g(t-s,y+|d_n(s)|)-\mathbf g(t-s,y)\right]^2dyds\right\}^p\\
 &\le K_1\mathbb E\left\{\int_0^t|d_n(s)|^2\int_0^\infty \left[\int_0^1 \frac{\partial \mathbf g}{\partial y}(t-s,y+\lambda |d_n(s)|)d\lambda\right]^2dyds\right\}^p\\
 &\le K_2\mathbb E\left\{\int_0^1\int_0^t|d_n(s)|^2\int_0^\infty \frac{(y+\lambda|d_n(s)|)^{2\alpha}}{(t-s)^3}\exp\left[-\frac{(y+\lambda|d_n(s)|)^2}{(t-s)}\right]dydsd\lambda\right\}^p\\
 &\quad\quad+K_3\mathbb E\left\{\int_0^1\int_0^t|d_n(s)|^2\int_0^\infty \frac{(y+\lambda|d_n(s)|)^{2\alpha+4}}{(t-s)^5}\exp\left[-\frac{(y+\lambda|d_n(s)|)^2}{(t-s)}\right]dydsd\lambda\right\}^p\\
 &\le K_4\mathbb E\left\{\int_0^t\frac{|d_n(s)|^2}{(t-s)^{\frac52-\alpha}}ds \right\}^p\le K_5\mathbb E\left\{\int_0^t\int_0^s\frac{s(\beta_n(r)-\beta_{n-1}(r))^2}{(t-s)^{\frac52-\alpha}}drds \right\}^p\\
 &=K_5\mathbb E\left\{\int_0^t\int_r^t\frac{s(\beta_n(r)-\beta_{n-1}(r))^2}{(t-s)^{\frac52-\alpha}}dsdr \right\}^p.
 \end{split}\]
 Note that for the above argument to work we must have $\alpha>\frac32$, otherwise the first integral with respect to $r$ is $\infty$. Since $\frac52-\alpha>-1$, we get that
 \[ H_n(t)\le K\int_0^t H_{n-1}(s)(t-s)^{\frac32-\alpha}ds.\]

Also, define
\[\begin{split}
\rho\dot\beta_0(t)&=\int_0^\infty \frac{\partial p_-}{\partial x}(t,0,y)u_0(y)dy \\
&\quad\quad +\int_0^t\int_0^\infty \tilde p(t-s,0,y)\frac{\sigma(y)}yW(dyds),
\end{split}\]
then from Lemma~\ref{Thm:sheatCalc} (1) we have
\[ \mathbb E|\dot\beta_0(t)|^{2p}<\infty\]

Therefore, by Walsh's Lemma 3.3 in~\cite{walsh86}, we obtain that $\beta(t)$ as the limit of $\beta_n(t)$ exists and is unique, and we also have
\[ \mathbb E|\dot\beta(t)|^{2p}<\infty\]
and its bound is uniform (that is, not dependent of $t$). $\Box$
\end{proof}

Let $\dot\beta(t)$ be the solution to~\eqref{Eq:SSPWBeta}. Fix $L>0$, and define $\Psi_L:[0,\infty)\to [0,1]$ as a smooth monotone decreasing function that satisfies $\chi_{[0,L]}\le\Psi_L\le\chi_{[0,L+1]}$. Now consider the following iteration of $\tilde u^L(t,x)$:
\begin{equation}\label{Eq:SSPWUL}
\begin{split}
\tilde u_n^L(t,x)&=\int_0^\infty p_-(t,x,y)u_0(y)dy\\
&\quad\quad+\int_0^t\int_0^\infty \dot\beta(s)\frac{\partial p_-}{\partial x}(t-s,x,y)\tilde u_{n-1}^L(s,y)\Psi_L(|\dot\beta(s)|)dyds\\
&\quad\quad+\int_0^t\int_0^\infty p_-(t-s,x,y)\sigma(y) W_\beta(dyds),
\end{split}\end{equation}
where the initial value
\[ \tilde u_0^L(t,x):=\int_0^\infty p_-(t,x,y)u_0(y)dy+\int_0^t\int_0^\infty p_-(t-s,x,y)\sigma(y) W_\beta(dyds).\]
Then we have
\begin{lemma} The solution $\tilde u^L(t,x)$ to~\eqref{Eq:SSPWUL} exists and is unique.
\end{lemma}
\begin{proof} By adopting a similar approach to the proof of Theorem~\ref{Thm:sheatMain}, we first compute
\[ \frac{\tilde u_{n+1}^L(t,x)-\tilde u_n^L(t,x)}x=\int_0^t\int_0^\infty \dot\beta(s)\tilde q(t-s,x,y)\left[\frac{\tilde u_n^L(s,y)-\tilde u_{n-1}^L(s,y)}y\right]\Psi_L(|\dot\beta(s)|)dyds.\]

Define
\[ J_n(t):=\sup_{x>0}\mathbb E\left[\left|\frac{\tilde u_{n+1}^L(t,x)-\tilde u_n^L(t,x)}x\right|^{2p}\right].\]
From Lemma~\ref{Thm:SSPWCalc} (2), still use Burkholder's inequality and Jensen's inequality and we get
\[\begin{split}
 J_n(t)&\le K_1(L+1)^{2p}\mathbb E\left\{\int_0^t\int_0^\infty \tilde q(t-s,x,y)\left[\frac{\tilde u_n^L(s,y)-\tilde u_{n-1}^L(s,y)}y\right]dyds\right\}^{2p}\\
 &\le K_2\int_0^t \frac{J_{n-1}(s)}{\sqrt{t-s}}ds.
\end{split}\]

Also, from Lemma~\ref{Thm:sheatCalc} (1) and $|\sigma(y)|\le By$, by following the same calculations as in the previous sections we get
\[ \sup_{x>0}\mathbb E\left[\left|\frac{\tilde u_0^L(t,x)}x\right|^{2p}\right]<\infty\]

Therefore, by Walsh's Lemma 3.3 in~\cite{walsh86}, we obtain that $\tilde u^L(t,x)$ as the limit of $\tilde u_n^L(t,x)$ exists and is unique, and we also have
\[ \sup_{x>0}\mathbb E\left[\left|\frac{\tilde u^L(t,x)}x\right|^{2p}\right]<\infty\]
and its bound is uniform (that is, not dependent of $t$). $\Box$
\end{proof}

\subsubsection{The Stefan Boundary Condition Holds}
\begin{lemma}\label{Thm:SSPWB}
Let $\tilde u^L$ and $\beta$ be the unique solutions to~\eqref{Eq:SSPWUL} and~\eqref{Eq:SSPWBeta}. Then $\mathbb P$-a.s., the Stefan boundary condition holds, namely,
\[ \lim_{x\searrow 0}\frac{\tilde u^L(t,x)}x=\rho\dot\beta(t),\forall 0\le t<\tau^L
\]
and $\dot\beta(t)$ is $\left(\frac14-\epsilon\right)$-H\"{o}lder continuous.
\end{lemma}
\begin{proof}
This lemma is shown by using Kolmogorov Continuity Theorem and adopting the similar calculations as in the previous section. In fact, define for $x>0$ $v(t,x):=\tilde u^L(t,x)/x$ and $v(t,0):=\rho\dot\beta(t)$, then we claim that $\mathbb P$-a.s., $v(t,x)$ is $\left(\frac14-\epsilon,\frac16-\epsilon\right)$-H\"{o}lder continuous for fixed $T>0$ on $[0,T]\times[0,1]$ with $\epsilon>0$.

Indeed, by Burkholder's inequality,
\[\begin{split}
\mathbb E\left[|v(t,x+h)-v(t,x)|^{2p}\right]&\le K_1\mathbb E\left[\int_0^t\int_0^\infty [\tilde q(t-s,x+h,y)-\tilde q(t-s,x,y)]v(s,y)dyds\right]^{2p}\\
&\quad\quad+K_2\mathbb E\left[\int_0^t\int_0^\infty [\tilde p(t-s,x+h,y)-\tilde p(t-s,x,y)]^2v^2(s,y)dyds\right]^p.
\end{split}\]

By Lemma~\ref{Thm:sheatCalc} (2) and Lemma~\ref{Thm:SSPWCalc} (2) about $\tilde p$ and $\tilde q$, we have that by using Jensen's inequality as in the previous section, for $0\le x\le 1$,
\[ \mathbb E\left[|v(t,x+h)-v(t,x)|^{2p}\right]\le Kh^\frac{p}3\]
where $K$ does not depend on $t$. Also, by Lemma~\ref{Thm:sheatCalc} (3) and Lemma~\ref{Thm:SSPWCalc} (3), we have that still by using Jensen's inequality, for $0\le t\le T$,
\[ \mathbb E\left[|v(t+k,x)-v(t,x)|^{2p}\right]\le Kk^\frac{p}2.\]

By Kolmogorov's Continuity Theorem, we have that for all $\delta<1/4,\gamma<1/6$, $v(t,x)$ is $\mathbb P$-a.s. uniformly $(\delta,\gamma)$-H\"{o}lder continuous, which implies that $\mathbb P$-a.s.,
\[ \lim_{x\searrow 0} v(t,x)=\rho\dot\beta(t), \forall t\in[0,T],\]
and $\dot\beta(t)=v(t,0)/\rho$ is $\left(\frac14-\epsilon\right)$-H\"{o}lder continuous.

Now, when $\tau^L<\infty$, we simply let $T:=\tau^L$. When $\tau^L=\infty$, we choose $T:=n$ for all $n\in\mathbb N$, and since the above statement holds for $t\in[0,n]$ for all $n\in\mathbb N$ we have that it holds for $t\in[0,\infty)$. $\Box$
\end{proof}

\subsubsection{Relaxation of the Truncation}
\begin{lemma} Define $\tau:=\lim_{L\to\infty}\tau^L$ and $\tilde u(t,x):=\lim_{L\to\infty}\tilde u^L(t,x)$. Then $\tilde u(t,x)$ is the unique solution to~\eqref{Eq:SSPW}.
\end{lemma}
\begin{proof}
 Fix $t\in[0,\tau)$ and define $L_t:=\sup\{|\dot\beta(r)|:0\le r\le t\}$. Then we have $L_t<\infty$. Therefore for $0\le r\le t$ we have that $\tilde u(r,x)$ also satisfies the equation of $\tilde u^{L_t}(r,x)$. By the uniqueness of $\beta$ and $\tilde u^{L_t}$, this implies that for $0\le r\le t$,
\[ \tilde u(r,x)=\tilde u^{L_t}(r,x)\textrm{ and }\lim_{x\searrow 0} \frac{\tilde u(r,x)}x = \rho\dot\beta(r).\]

Therefore, for $0\le t< \tau$, $\tilde u(t,x)$ satisfies~\eqref{Eq:SSPW} and also
\[ \lim_{x\searrow 0} \frac{\tilde u(t,x)}x = \rho\dot\beta(t). \Box\]
\end{proof}

\subsection{Numerical Simulation}
Since the existence and uniqueness of the problem is shown we can indeed simulate the solution and its boundary numerically. The simulation uses finite-difference Euler approximation scheme described in~\cite{numschemes}. To guarantee numerical robustness in our simulation we assume the space and time increment steps satisfy $\Delta t<(\Delta x)^2/2$.

In this simulation we apply the following simulation parameters:
\[\begin{split}
\rho&=-0.2;\\
u_0(x)&=\frac{x+x^2}{1+x^4/16};\\
\sigma(x)&=\frac{x^2}{1+4x}.
\end{split}\]

The consequent numerical results are illustrated as follows.
\begin{enumerate}
\item[(1)] Figure~\ref{Fig:SSPWU} illustrates the weak solution $u(t,x)$;
\item[(2)] Figure~\ref{Fig:SSPWB} illustrates the boundary derivative $\dot\beta(t)$;
\item[(3)] Figure~\ref{Fig:SSPWUU} illustrates the typical shape of the solution $u(t,x)$ at a particular time $t>0$, from which we see that $u$ is smoother as $x$ gets closer to the boundary (shifted and denoted by $0$), and is differentiable at the boundary.
\end{enumerate}

\begin{figure}
\centering
\includegraphics[width=110mm]{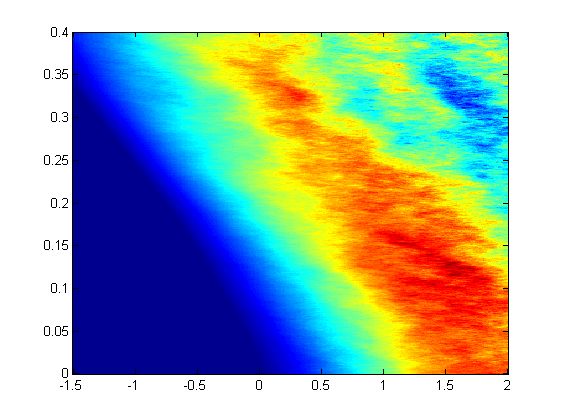}
\caption{Weak solution $u(t,x)$}
\label{Fig:SSPWU}
\end{figure}

\begin{figure}
\centering
\includegraphics[width=110mm]{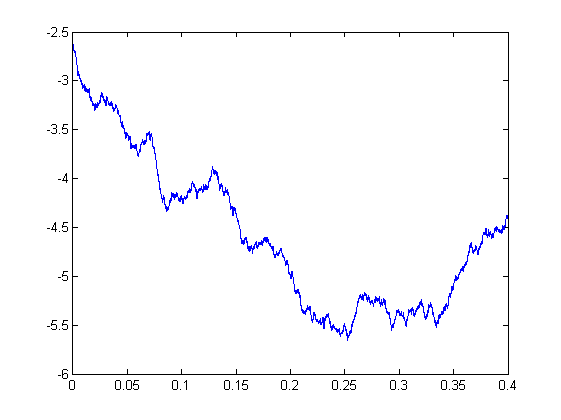}
\caption{Boundary derivative (or speed of moving boundary) $\dot\beta(t)$}
\label{Fig:SSPWB}
\end{figure}

\begin{figure}
\centering
\includegraphics[width=110mm]{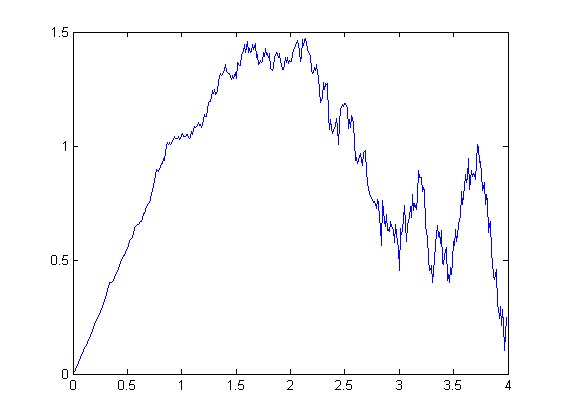}
\caption{A typical shape of the solution $u$ at some time $t>0$. $u$ is smoother as $x$ gets closer to the boundary (shifted and denoted by $0$), and is differentiable at the boundary.}
\label{Fig:SSPWUU}
\end{figure}

\section{Summary}

In this section we summarize the main results obtained in the previous sections. Throughout this section we fix a probability space $(\Omega,\mathcal F,\mathbb P)$, and suppose $W:\Omega\times\mathbb R\times\mathbb R\to\mathbb R$ is a standard 2-dimensional Brownian sheet. In this paper we studied 3 types of stochastic equations.

\subsection{A Stochastic Stefan Problem with Spatially Colored Noise}

The solution $u(t,x)$ to the stochastic Stefan equation
\begin{equation}
\begin{split}
 \frac{\partial u}{\partial t}&=\frac{\partial^2 u}{\partial t^2}+ u(t,x)d\xi_t(x),\forall x>\beta(t),\\
u(t,x)&=0,\forall x\le \beta(t),\\
u(0,x)&=u_0(x), \\
\rho \dot \beta(t) &=\lim_{x\searrow \beta(t)} \frac{u(t,x)}x, \mathbb P\textrm{-a.s.}\\
\end{split}
\end{equation}
 exists and is unique for $0\le t <\tau:=\lim_{L\to\infty}\tau^L$ where $\tau^L:=\inf\{t\in\mathbb R_+\cup\{0\}: |\dot \beta(t)|\ge L\}$.

\subsection{Boundary Regularity of the Stochastic Heat Equation}

The solution $u(t,x)$ to the stochastic heat equation
\begin{equation}
\begin{split}
\frac{\partial u}{\partial x}(t,x)&=\frac{\partial^2 u}{\partial x^2}(t,x)+u(t,x)\frac{\partial^2 W}{\partial t\partial x},\forall x>0,t\in[0,T],\\
u(0,x)&=u_0(x),\\
u(t,0)&=0,\forall x\le 0,t\in[0,T]
\end{split}\end{equation}
exists and is unique; Moreover, define $v(t,x):=u(t,x)/x$, then $\mathbb P$-a.s.,
\[ \lim_{x\searrow 0}v(t,x)=\frac{\partial u}{\partial x}(t,0+)=\int_0^\infty \frac{\partial p}{\partial x}(t,0,y)u_0(y)dy + \int_0^t\int_0^\infty \frac{\partial p}{\partial x}(t,0,y) u(s,y)W(dyds);\]
If we further define $v(t,0):=\lim_{x\searrow0}v(t,x)$, then $\mathbb P$-a.s.,
$v(t,x)$ is $\left(\frac14-\epsilon,\frac16-\epsilon\right)$-H\"{o}lder continuous on $[0,T]\times[0,1]$ for $\epsilon>0$.

\subsection{A Stochastic Stefan Problem with Space-Time Brownian Noise}

Let $\sigma:[0,\infty)\to\mathbb R$ be a regular scaling function (see the definition given in Section~\ref{Ch:sspw}). Then the solution $u(t,x)$ to the stochastic Stefan equation
\begin{equation}
\begin{split}
 \frac{\partial u}{\partial t}&=\frac{\partial^2 u}{\partial t^2}+ \sigma(x-\beta(t))\frac{\partial^2 W}{\partial t\partial x},\forall x>\beta(t),\\
u(t,x)&=0,\forall x\le \beta(t),\\
u(0,x)&=u_0(x), \\
\rho \dot \beta(t) &=\lim_{x\searrow \beta(t)} \frac{u(t,x)}x,\mathbb P\textrm{-a.s.}\\
 \end{split}\end{equation}
exists and is unique for $0\le t<\tau$ for $0\le t <\tau:=\lim_{L\to\infty}\tau^L$ where $\tau^L:=\inf\{t\in\mathbb R_+\cup\{0\}: |\dot \beta(t)|\ge L\}$; Moreover, $\mathbb P$-a.s., $\beta\in C^1([0,\tau))$ and $\dot\beta(t)$ is $\left(\frac14-\epsilon\right)$-H\"{o}lder continuous for $\epsilon>0$.

\bibliographystyle{plain}
\bibliography{thesisbib}

\end{document}